\documentclass{article}
\usepackage[pdftex]{hyperref}
\usepackage{tikz-cd, url}

\hypersetup{
  colorlinks   = true, 
  urlcolor     = gray, 
  linkcolor    = red, 
  citecolor   =  blue
}
\usepackage{amsmath,amsfonts,amsthm,amssymb,graphicx,xcolor}
\usepackage[all,2cell,ps]{xy}

\theoremstyle{plain}
\newtheorem{thm}{Theorem}[section]
\newtheorem{lem}[thm]{Lemma}
\newtheorem{prop}[thm]{Proposition}
\newtheorem{cor}[thm]{Corollary}
\newtheorem{conj}[thm]{Conjecture}

\theoremstyle{definition}
\newtheorem{defn}[thm]{Definition}

\newtheorem{ex}[thm]{Example}

\newtheorem{rem}[thm]{Remark}

\newcommand{\Z}{\mathbb Z}

\newcommand{\C}{\mathbb C}

\newcommand\sO{{\mathcal O}}

\newcommand\sH{{\mathcal H}}

\newcommand\sF{{\mathcal F}}
\newcommand\sG{{\mathcal G}}

\newcommand\sI{{\mathcal I}}

\DeclareMathOperator{\Pic}{Pic}

\DeclareMathOperator{\Alb}{Alb}

\newcommand{\codim}{{\rm codim}\,}

\usepackage{mathtools}

\DeclarePairedDelimiter\abs{\lvert}{\rvert}
\DeclarePairedDelimiter\norm{\lVert}{\rVert}

\makeatletter
\let\oldabs\abs
\def\abs{\@ifstar{\oldabs}{\oldabs*}}
\let\oldnorm\norm
\def\norm{\@ifstar{\oldnorm}{\oldnorm*}}
\makeatother
\newcommand{\bigslant}[2]{{\raisebox{.2em}{$#1$}\left/\raisebox{-.2em}{$#2$}\right.}}

\title{$L^2$-Betti Numbers and Convergence of Normalized Hodge Numbers via the Weak Generic Nakano Vanishing Theorem}
 \author{\small{Luca F. Di Cerbo} \\ \scriptsize{University of Florida} \\ \footnotesize{\textsf{ldicerbo@ufl.edu}} \and 
\small{Luigi Lombardi}\footnote{Partially supported by the Rita Levi-Montalcini Program, Simons Foundation, SIR 2014:  AnHyC: “Analytic aspects in complex and hypercomplex
geometry” (code RBSI14DYEB), and  Grant 261756 of the Research Councils of Norway.} \\ 
\scriptsize{University of Milan}\\ \footnotesize{\textsf{luigi.lombardi@unimi.it}}}
\date{}
	
\begin{document}

\maketitle

\begin{abstract}
 We study the rate of growth of normalized Hodge numbers along a tower of abelian covers of a smooth projective variety with semismall Albanese map. These bounds are in some cases  optimal.  Moreover, we compute the $L^2$-Betti numbers of irregular varieties that satisfy the weak generic Nakano vanishing theorem (\emph{e.g.}, varieties with semismall Albanese map).  Finally, we study the convergence of normalized plurigenera along towers of abelian covers of any irregular variety.   As an application, we extend a result of Koll\'ar concerning the multiplicativity of  higher plurigenera of a smooth projective variety of general type, to a wider class of varieties. In the Appendix, we study irregular varieties for which the first Betti number diverges along a tower of abelian covers induced by the Albanese variety. 
 
\end{abstract}
\vspace{6cm}
\tableofcontents\quad\\

\vspace{1cm}

\section{Introduction and Main Results}

In this paper, we study the asymptotic behavior of Hodge and Betti numbers of sequences of coverings of complex projective varieties with semismall Albanese map. Similar problems have attracted considerable interest over the last four decades, and they have been extensively studied in a variety of different geometric contexts. For instance, in \cite{Wallach} and \cite{Wallach1}, DeGeorge and Wallach  study the asymptotic behavior of Betti numbers on regular coverings of compact locally symmetric spaces of non-compact type (\emph{e.g.}, compact real and complex hyperbolic manifolds). The problem addressed in \cite{Wallach, Wallach1} is natural for researchers interested in the cohomology of locally symmetric varieties, and it can be easily described.  In what follows, we rephrase the main result in \cite{Wallach, Wallach1} in terms of \emph{normalized} Betti numbers. We refer to page 714 in the introduction of paper \cite{Bergeron}, or to Chapter 5 in the book \cite{Luck02}, for more details concerning the connections between the representation theoretic results of DeGeorge and Wallach, and the asymptotic properties of the cohomology of compact locally symmetric varieties.

Given a torsion free lattice $\Gamma$ acting co-compactly on a symmetric space of non-compact type, say $G/ K$, a sequence of nested, normal, finite index subgroups $\{\Gamma_i\}$ of $\Gamma$ is a cofinal filtration of $\Gamma$ if $\cap_i\Gamma_i$ is the identity element. Define $\pi_i\colon X_i\rightarrow X$ as the finite index regular cover of $X\stackrel{{\rm def}}{=}\Gamma \backslash (G/ K)$ associated to $\Gamma_i$. The main result of \cite{Wallach} implies that
\[
\lim_{i \to \infty}\frac{b_{k}(X_i)}{\deg \pi_i } \; = \; 0 \quad  \mbox{ for  any }\quad  k\neq \frac{1}{2}\dim (G/K),
\] 
where $b_k(X_i)$ denotes the $k$-th Betti number of $X_i$. We refer to the ratio $b_k(X_i)/\deg \pi_i$ as the \emph{normalized} $k$-Betti number of the cover $\pi_i\colon  X_i \rightarrow X$. Thus, for $k$ different from the middle dimension, the growth of Betti numbers in a tower of coverings associated to a cofinal filtration has sub-degree (or sub-volume) growth, and the normalized Betti numbers converge to zero. 

The study of Betti numbers in a sequence of coverings continues to fascinate many mathematicians; see for example the recent work of Abert \emph{et al.} \cite{Bergeron}. In this remarkable paper, the authors extend the results of DeGeorge--Wallach to sequences of compact locally symmetric varieties which Benjamini--Schramm converge to their universal covers.
We refer to \cite{Bergeron} for the precise definition of this notion of convergence; here we simply remark that a tower of coverings associated to a cofinal filtration does indeed Benjamini--Schramm converge. The techniques employed both in \cite{Wallach, Wallach1} and \cite{Bergeron} are based on representation theory, and they do not immediately generalize to non-symmetric varieties. Nevertheless, there is a large and growing literature concerning these kind of problems outside the locally symmetric context; see for example \cite{Yeung1}, \cite{DS17}, \cite{Bergeron2} and the bibliography therein. These papers employ geometric analysis techniques, and they extend much of the DeGeorge--Wallach theory to negatively curved compact Riemannian manifolds which are \emph{not} locally symmetric.
 
Here we contribute to this circle of ideas by studying the cohomology of complex projective varieties with \emph{semismall} Albanese map, a further instance of varieties of non-locally symmetric type. Our approach is based on tools of Algebraic Geometry and Hodge Theory, and it employs sheaf-theoretic techniques specific to this class of varieties. As an important ingredient, we employ the generic vanishing theory of bundles of holomorphic $p$-forms developed by Popa and Schnell in \cite{PS13} via Saito's theory of mixed Hodge modules. 

We now turn to details and present our main results. Let $X$ be an irregular smooth projective complex variety of dimension $n$, and let $a_X\colon X \to \Alb(X)$ be its Albanese map. The Albanese torus $\Alb(X)$ is an abelian variety of dimension $g=h^{1, 0}(X)$ (we recall that the variety $X$ is irregular if $g >0$). We say that the Albanese map $a_X$ is \emph{semismall} if for every integer  $k>0$ the following inequalities hold
\begin{equation}\label{ssmall}
{\rm codim} \{\, x \in a_X(X) \, | \, \dim \big( a_X^{-1}(x) \big)  \geq k \, \} \; \geq 2k.
\end{equation}
In particular, if $a_X$ is semismall, then $a_X$ is generically finite onto its image, but the converse does not hold in general.  For instance, the Albanese map of the blow-up of an abelian variety  along a smooth subvariety of codimension $c$ is semismall if and only if $c\leq 2$.
Next, let 
\[
\mu_d \colon \Alb(X) \to \Alb(X), \quad \mu_d(x) \; = \; dx \; = \; \overbrace{x+\cdots + x}^{d\mbox{-times}}, \quad d\geq 1
\]
be the multiplication maps on $\Alb(X)$, and define the varieties
$X_d$ via the fiber product diagrams
\begin{equation}\label{diagrintro}
\centerline{ \xymatrix@=32pt{
 X_d \ar[d]^{\varphi_d} \ar[r]^{a_d\,\,\,\,\,\,\,} & \Alb(X) \ar[d]^{\mu_d}  \\
 X  \ar[r]^{a_X\,\,\,\,\,\,\,}  & \Alb(X).\\}} 
 \noindent 
\end{equation}
Our first result controls the rate of growth of the Hodge numbers of $X_d$ with respect to the degrees of the covers $\varphi_d \colon  X_d\rightarrow X$. We refer to the ratios $h^{p, q}(X_d) / \deg \varphi_d$ as the \emph{normalized} $(p, q)$-Hodge numbers.
The following theorem  provides an \emph{effective} estimate for the rate of convergence
of the normalized Hodge numbers,  and it also yields the optimal rate of convergence of one of them.

\begin{thm}\label{main1}
	Let $X$ be a smooth projective variety of complex dimension $n$, and let $\varphi_d \colon X_d \to X$ be the \'etale covers defined in 	\eqref{diagrintro}. If the Albanese map $a_X$ is semismall, then
	for any pair of integers $(p,q) \in  [0,n]^2$ there exists a positive constant $B(p,q)$ such that 
	\begin{equation}\label{degreeh}
	\frac{h^{p,q}(X_d)}{\deg \varphi_d } \; \leq \; B(p,q) \, d^{ \, -2 \abs{ n-p-q }   } \quad \mbox{for all} \quad d\geq 1.
	\end{equation}
	Moreover, we have 	
\begin{equation}\label{limhodge}	
	\lim_{d\to \infty} \frac{h^{p,q}(X_d)}{\deg \varphi_d } \;  = \;  (-1)^q  \; \chi(\Omega_X^p) \quad \mbox{if} \quad p+q=n.
	\end{equation}
Conversely, if $X$ is a smooth projective variety of dimension $n$ that satisfies both  $\dim \Alb(X) > n$ and the bounds in \eqref{degreeh} for 
all pairs of indexes $(p,q)\in [0,n]^2$, then the Albanese map $a_X$ is semismall.
	\end{thm}
	
In order to prove the previous theorem, in Section \ref{secasy} we develop a general machinery that 
establishes the convergence of the  normalized cohomology ranks 
$h^q(X_d , \varphi_d^* \sF ) / \deg \varphi_d$ of a coherent sheaf $\sF$ on $X$ subject to certain cohomological conditions (\emph{cf}. Theorem \ref{cor1}). In particular, Theorem \ref{main1} corresponds to the case of bundles of holomorphic $p$-forms $\sF = \Omega_X^p$. In Section \ref{secplurigenera}, we  apply this machinery to the case of pluricanonical bundles $\sF = \omega_X^{\otimes m}$ for $m\geq 1$.
We refer to Section \ref{secbetti} for the details of the proof of Theorem \ref{main1}, and to a  generalization that takes into account all values of the defect of semismallness of the Albanese map (\emph{cf}.  Theorem \ref{thmdefect} and \eqref{defect}).
Finally,  Theorem \ref{main1} implies the following statement regarding the normalized Betti numbers. 

\begin{cor}\label{maincor1}
Let $X$ be a smooth projective variety of dimension $n$ such that the Albanese map $a_X$ is semismall. Then for any integer $k\neq n$ there exists a positive constant $C(k)$ such that
\begin{equation*}\label{degreeb}
\frac{b_k(X_d)}{\deg \varphi_d } \; \leq \; C(k) \, d^{ \, -2 \abs{n-k}}  \quad \mbox{for all} \quad d\geq 1.
\end{equation*}
Furthermore, we have 
$$\lim_{d\to \infty} \frac{b_n(X_d)}{\deg \varphi_d } \;  = \;  (-1)^n  \; \chi_{\rm top}(X).$$
\end{cor}
 
When combined with L\"uck's Approximation Theorem (\emph{cf}. \cite[Main Theorem]{Luck}), Theorem \ref{main1} can be used to compute the $L^2$-Betti numbers of the Albanese universal cover $\overline{\pi} \colon \overline{X} \rightarrow X$, when the Albanese map of $X$ is semismall. Throughout the paper, the \emph{Albanese universal cover} is defined as the pullback
 of $a_X$ via the universal topological cover $\pi$ of $\Alb(X)$:

\begin{equation}\label{diagrintro2}
\centerline{ \xymatrix@=32pt{
 \overline{X} \ar[d]^{\overline \pi} \ar[r]^{\overline a\,\,\,\,\,\,\,} & \C^g \ar[d]^{\pi}  \\
 X  \ar[r]^{a_X\,\,\,\,\,\,\,}  & \Alb(X), \\}} 
 \noindent 
\end{equation}
where $g=h^{1, 0}(X) = \dim \Alb(X)\neq 0$. Notice that, up to a finite cover, the Albanese universal cover coincides with the universal abelian cover. Indeed, these infinite covers are equal if and only if $H_1(X, \Z)$ is torsion free. We refer to Section \ref{secl2} for the formal definition of $L^2$-Betti numbers of any infinite $G$-covering map $X' \rightarrow  X' / G$.  It turns out that our calculation of $L^2$-Betti numbers holds 
for a  more general class of smooth irregular projective varieties, which we now define.  
We say that $X$ satisfies the \emph{weak generic Nakano vanishing theorem}
if for any pair of integers  $(p,q) \in [0,n]^2$ such that $p+q \neq n$ we have
$$ H^q\big( X ,\Omega_X^p \otimes \alpha_{p,q}  \big) \; = \; 0$$
for at least one  topologically trivial line bundle $\alpha_{p,q} \in \Pic^0(X)$.
Instances of varieties that satisfy this property  are varieties with semismall Albanese map (\emph{cf}. Theorem \ref{psthm}), and varieties that admit one holomorphic $1$-form such that its zero-set is either finite of empty (\emph{cf}. Theorem \ref{glthm}). We refer to \cite{gl:gv1} and  \cite[Introduction, Sections 3.1 and 3.2]{L13} for examples and basic properties of this class of varieties. 

\begin{thm}\label{main2}
	Let $X$ be a smooth projective variety of complex dimension $n$  and let $\overline X$ be the universal Albanese cover. 
If 	$X$ satisfies the weak generic Nakano vanishing theorem, then the $L^2$-Betti numbers of $\overline X$ are:
		\begin{equation*}
	b^{(2)}_k (\overline{X})=\begin{cases} 
	(-1)^n \chi_{\rm top}(X)  &\mbox{if }\quad    k =  n  \\
	0 & \mbox{if }\quad   k  \neq n.
	\end{cases}
	\end{equation*}
\end{thm}

It is tantalizing to compare Theorem \ref{main2} with an old conjecture of Singer concerning the $L^2$-Betti numbers of the universal covering space of an aspherical manifold.

\begin{conj}[Singer Conjecture]
If $X$ is a closed aspherical manifold of real dimension $2n$, then
\begin{equation*}
b^{(2)}_{k}(\widetilde{X})=\begin{cases} 
(-1)^n \chi_{\rm top}(X)  &\mbox{if }\quad    k =  n  \\
0 & \mbox{if }\quad   k  \neq n,
\end{cases}
\end{equation*}
where $\pi \colon \widetilde{X}\rightarrow X$ is the topological universal cover of $X$.
\end{conj}

Interestingly, Theorem \ref{main2} provides a vanishing theorem analogous \sloppy to Singer's conjecture when the $L^2$-Betti numbers are computed with respect to the Albanese universal cover. It seems worth asking  whether Theorem \ref{main2} holds when the $L^2$-Betti numbers are computed with respect to the topological universal cover, and more generally, if Singer's conjecture can be extended meaningfully outside the class of aspherical manifolds, at least within the class of  projective varieties. 
 
We point out that  in \cite[Theorem 3 (i)]{JZ}  Jost and Zuo prove, among other things, a special case of Theorem \ref{main2}. More specifically, they prove Theorem \ref{main2} in the case of  smooth projective varieties whose  Albanese map $a_X\colon X\rightarrow \Alb(X)$ is an immersion. The techniques used by Jost and Zuo rely on analytical arguments introduced by Gromov in \cite{Gro}, where the author confirms Singer's conjecture for \emph{K\"ahler hyperbolic} manifolds. These manifolds include K\"ahler manifolds with negative and pinched sectional curvature, and do not  \emph{not} contain any rational curve. 
On the contrary, varieties with Albanese map semismall may contain rational curves. Finally, we remark that in \cite{LMW17b} the semismallness condition of the Albanese map is studied by means of topological generic vanishing theory.  Via the general strategy of  \cite[Theorem 2.28]{LMW17a}, it is possible that the techniques of \cite[Theorem 1.2]{LMW17b} suffice to give an alternative proof of Theorem \ref{main2} in the case of varieties with  semismall  Albanese map; however we do not pursue this direction in this paper. Finally, we also point out the related work of Budur \cite{budur}, where the author shows polynomial periodicity of the Hodge numbers of congruence covers.

In Section \ref{secplurigenera}, we apply the techniques of Section \ref{secasy} to prove a version of Theorem \ref{main1} for pluricanonical bundles $\omega_X^{\otimes m}$ with $m\geq 2$. More precisely, we compute the following limits 
\begin{equation}\label{genera}
\widetilde{P}_m(X) \; \stackrel{{\rm def}}{ =  } \;  \lim_{d\to \infty} \frac{P_m(X_d)}{ \deg \varphi_d}    \; = \;  \lim_{d\to \infty} \frac{h^0(X_d,\omega_{X_d}^{\otimes m})}{ \deg \varphi_d }, \quad m\geq 2
\end{equation}
of normalized plurigenera  (whenever they exist).
Let $I\colon X \to Z$ be a smooth representative of the Iitaka fibration, and let  $q(I) = q(X) - q(Z)$ be the difference of the irregularities. In Proposition \ref{plurinu}, we prove that the limits in \eqref{genera} exist and  are computed by:
\begin{equation}\label{intropm}
\widetilde{P}_m(X) \;  = \;  \begin{cases} 
P_m(X) &\mbox{if }\quad    q(I) = 0 ,  \\
0 & \mbox{if }\quad   q(I) > 0.
\end{cases}
\end{equation}
We recall that if $X$ is of general type (hence satisfying $q(I)=0$), then a classical result of Koll\'ar \cite[Proposition 9.4]{KP} (\emph{cf}. also \cite[Theorem 11.2.23]{Laz2}) ensures  that its higher plurigenera  are  multiplicative with respect to \emph{any} \'etale cover. As suggested by \eqref{intropm}, we extend this property to smooth projective varieties satisfying  $q(I)=0$, when the \'etale covers are induced  by the Albanese variety via base change. Also, as a by-product, we show that \cite[Proposition 9.4]{KP} cannot be extended to varieties with $q(I)>0$. We refer to Section \ref{secplurigenera} for the proof of Theorem \ref{multiplication}, and examples of varieties with $q(I)=0$ that are not of general type.\\

In the Appendix (Section \ref{unbounded}), we discuss the irregular varieties  for which the first Betti number $b_1$  goes to infinity along the unramified covers induced by the multiplication  maps on the Albanese variety (regardless of the semismallness of the Albanese map). Building upon results of Beauville \cite{B}, we prove that if this is the case, then the base variety must be fibered over a curve having either genus at least two, or genus equal to one and the fibration admits  two multiple fibers whose multiplicities are not coprime.  Moreover, if the group $H^2(X, \Z)$ is torsion free, then the converse of this result holds as well (\emph{cf}. Theorem  \ref{qfibr}). We remark that, in many interesting cases, the converse   can be used to deduce that the first Betti number is indeed uniformly bounded on these abelian covers. 
Very recently, Stover \cite{Sto17} and Vidussi \cite{Vid17} study the boundedness of the first Betti number of abelian covers of the Cartwright--Steger surface \cite{CS}.  While our analysis does not fully recover their theorems,  it has the advantage to  put in perspective  their results in the framework of higher-dimensional varieties.\\


\noindent\textbf{Acknowledgments}. 
The authors thank  Rob Lazarsfeld and Christian Schnell for their interest and suggestions, and  Stefano Vidussi and Botong Wang for email correspondence.
They thank the Mathematics Department of Stony Brook University for the ideal research environment they enjoyed at the beginning of this project.
The first named author    thanks the University of Florida where most of these ideas crystallized, and where some of the results were presented in the Fall  2018 Topology Seminar.  During the preparation of this work, the second named author visited  both the University of Florida and  the Max Planck Institute for Mathematics  in Bonn. He thanks these institutions for their hospitality and support. Finally, he thanks Daniele Angella  of the University of Florence for financial support.\\ \\

\section{Weak $GV$-Sheaves}\label{secloci}
In this section, we recall a few basic results from generic vanishing theory.  The following presentation is tailored to our purposes; we refer to \cite{gl:gv1, gl:gv2, pareschi+popa:gv,Sch} for a comprehensive introduction.

Let  $X$ be a smooth projective complex variety  of dimension $n$, and let $f \colon X \to A$ be a morphism to an abelian variety of dimension $g$.  The \emph{non-vanishing loci} attached to a coherent sheaf $\sF$  on $X$ relative to $f\colon X \to A$ are defined as
$$V^i (\sF) \;  = \; \big\{ \, \alpha \in \widehat{A} \; \big| \;  H^i(X,   \sF \otimes f^*\alpha ) \neq 0  \, \big\}\quad (i\geq 0)$$
(in the notation $V^i(\sF)$ we omit the reference to the morphism $f$). 
Here $\widehat A \simeq \Pic^0(A)$ denotes the dual torus of $A$, which parameterizes 
  isomorphism classes of holomorphic line bundles  with trivial first Chern class. 
By the Semicontinuity Theorem \cite[Theorem III.12.8]{Ha}, the loci $V^i(\sF)$ are algebraic closed subsets of $\widehat A$.

\begin{defn}\label{defgv}
The sheaf $\sF$ satisfies $GV$ (or the \emph{generic vanishing property})   if 
$\codim_{\widehat A} V^i(\sF)  \geq  i$  for all  $i>0$. 
\end{defn}

 A fundamental result of Green and Lazarsfeld proves
 that if the Albanese map $a_X \colon X \to \Alb(X)$ is generically finite onto its image, then the canonical bundle $\omega_X$ satisfies $GV$. 
 Moreover,  the loci 
$V^i(\omega_X)$  are \emph{torsion linear} varieties for all $i\geq 0$ regardless the Albanese dimension of $X$, \emph{i.e.}  every irreducible component $T\subset V^i(\omega_X)$ is of type $\beta + T_0$ where $\beta \in \Pic^0(\Alb(X)) \simeq \Pic^0(X)$ is an element of finite order,  and $T_0 \subset \Pic^0(X)$ is a subtorus  (\emph{cf}. \cite[Theorem 0.1]{gl:gv2}, \cite{Sim93} and \cite[Corollary 19.2]{Sch}).  
For the purposes of this paper, we will consider the following weaker notion of generic vanishing.
\begin{defn}
The sheaf  $\sF$ satisfies \emph{weak} $GV$ \emph{with index} $p$ if 
$V^i(\sF)\subsetneq \widehat A$ for all $i\neq p$.
\end{defn}
 Obviously, $GV$-sheaves satisfy weak $GV$  with index $0$. 
 We conclude this subsection with a useful result which we will use in Section \ref{secplurigenera}. The   Euler characteristic of a sheaf $\sF$ is defined as $\chi(\sF) = \sum_{i\geq 0} (-1)^i h^i(X, \sF)$.

\begin{lem}\label{lemchi}
If $\sF$ is a weak $GV$-sheaf  with index $p$, then   $$\chi(\sF)  \; = \;  (-1)^p \,  h^p(X,\sF\otimes f^*\alpha)$$ for a generic element  $\alpha \in \widehat A$.
In particular, $\chi(\sF)=0$ if $\sF$ is a weak $GV$-sheaf with respect to two distinct indexes.
\end{lem}

\begin{proof}
If $\alpha \in \widehat A$ is generic, then the cohomology groups $H^i(X,\sF\otimes f^*\alpha)$ vanish for all $i\neq p$. Since $\chi(\sF)$ is invariant under twists with line bundles in $\Pic^0(X)$, we find $\chi(\sF)= \chi( \sF \otimes f^*\alpha ) = (-1)^p h^p(X,\sF \otimes f^*\alpha)$. Moreover, if $\sF$ is a weak $GV$-sheaf with respect to  two distinct indexes, then all the loci $V^i(\sF)$ are proper subset of $\widehat A$, hence $\chi(\sF)=0$.
\end{proof}

\subsection{(Weak) Generic Nakano Vanishing Theorem}

Let $X$ be a smooth projective variety of dimension $n$, and let $a_X \colon  X \to \Alb(X)$ be its Albanese map. Moreover, denote by 
$\Omega_X^p \stackrel{{\rm def}}{ = } \wedge^p \Omega_X$ the bundle of holomorphic $p$-forms on $X$.
Following \cite[Definition 12.1]{PS13}, we say that $X$ satisfies the \emph{generic Nakano vanishing theorem} if $\codim_{\Pic^0(X)} V^q(\Omega_X^p) \geq \abs{p+q-n}$ for all indexes $p$ and $q$. In this paper, we consider varieties that satisfy  a  weaker  vanishing condition.

\begin{defn}\label{weaknakano}
The variety $X$ satisfies the \emph{weak generic Nakano vanishing theorem} if $\Omega_X^p$ is a weak $GV$-sheaf with index $n-p$ for all $p=0,\ldots ,n$.
\end{defn}

It turns out that  $X$ satisfies the generic Nakano vanishing theorem if and only if it satisfies  a condition on the dimension of the fibers of the Albanese map.
This goes as follows. Set $V_l \stackrel{{\rm def}}{=}  \{ \, y\in \Alb(X) \, | \, \dim a_X^{-1}(y)\geq l\, \}$ and define the  \emph{defect of semismallness} of $a_X$ as:
\begin{equation}\label{defect} 
 \delta(a_X) \; = \; \max_{l \in \mathbf{N}} \{  2l - n + \dim V_l \}.
 \end{equation}
 
 \begin{defn}
We say that $a_X$ is \emph{semismall} if $\delta(a_X)=0$. Equivalently, $a_X$ is semismall if the inequalities of \eqref{ssmall} are satisfied for all $k\geq 1$. 
\end{defn}

\begin{thm}[Popa--Schnell]\label{psthm}
If  $X$ is a smooth projective variety of dimension $n$, then 
$$\codim_{\Pic^0(X)} V^q(\Omega_X^p) \; \geq \; \abs{ p \, + \, q \, - \, n } \, - \, \delta(a_X)$$ for all $p\geq 0$ and $q\geq 0$.
Moreover, there exists a pair $(p,q)$ for which the equality is attained.
In particular, if  $a_X$  is semismall, then  $X$ satisfies the generic Nakano vanishing  theorem.
\end{thm}
The previous theorem appears  in \cite[Theorem 3.2]{PS13}, and it is proved by means of Saito's theory of mixed Hodge modules and the Fourier--Mukai transform. Besides varieties with semismall Albanese map, another class of varieties that satisfies Definition \ref{weaknakano} is provided by the following result of Green and Lazarsfeld \cite[Theorem 3.1]{gl:gv1}.

\begin{thm}[Green--Lazarsfeld]\label{glthm}
Let $X$ be a smooth projective variety. If $X$  carries a holomorphic $1$-form such that its zero-set is either finite or empty, then $X$  satisfies the weak generic Nakano vanishing theorem.
\end{thm}

The previous theorem relies  on the deformation theory of the derivative complexes associated to  $\Omega_X^p$.
A natural question is the characterization of  varieties that satisfy Definition \ref{weaknakano}.
Here we note that a variety that satisfies  the weak generic Nakano vanishing theorem does not necessarily carry a holomorphic $1$-form whose zero-set is either finite or empty.  For instance, consider a smooth projective variety $Y$  of general type such that its Albanese map is an immersion and  $\codim_{\Alb(Y)} Y =2$ (for instance a genus $3$ curve in its Jacobian). Then the blow-up $Z$ of $\Alb(Y)$ along $Y$ is the counterexample we are looking for. In fact, by \cite[Theorem 2.1]{psform}, any holomorphic $1$-form on $Y$ has at least one zero, and its pull-back to $Z$ vanishes along some curves in the exceptional divisor. Moreover, all $1$-forms of $Z$ are obtained in this way as 
$H^0(Z,\Omega_Z) = H^0(\Alb(Y),\Omega_{\Alb(Y)}) = H^0(Y, \Omega_Y)$. On the other hand, the Albanese map of $Z$ is semismall so that $Z$ satisfies the generic Nakano vanishing theorem.

\section{Limits of Normalized Cohomology Ranks}\label{secasy}
 
Let $X$ be a smooth projective variety of complex dimension $n$, and $f\colon X \to A$ be a morphism to an abelian variety, as in Section \ref{secloci}. Given any integer $d\geq 1$, we denote by  $\mu_d\colon A\to A$ the multiplication map $\mu_d(x)=dx$. Furthermore, by means of the fiber product construction, we define the varieties  $X_d$ as follows:
\begin{equation}\label{fundamental}
\centerline{ \xymatrix@=32pt{
 X_d \ar[d]^{\varphi_d} \ar[r]^{f_d} & A \ar[d]^{\mu_d}  \\
 X  \ar[r]^{f}  & A .\\}} 
 \noindent
\end{equation}
In general the varieties $X_d$ may be disconnected, but if $f=a_X$ is the Albanese map they are irreducible.
Finally, we  set  $ \sF_d \stackrel{{\rm def}}{=} \varphi_d^* \sF$ if $\sF$ is a coherent sheaf on $X$. In this section we aim to calculate the following limits of \emph{normalized cohomology ranks}:
\begin{equation*}\label{limitscoh}
\liminf_{d\to \infty} \frac{h^p(X_d, \sF_d)} {\deg \varphi_d } \quad \quad \mbox{and} \quad \quad \limsup_{d\to \infty} \frac{h^p(X_d, \sF_d)} {\deg \varphi_d }.
\end{equation*}
To this end, we introduce first  some notation.
We denote by $r_i$ the number of irreducible components of $V^i(\sF)$, and by $v_i$ the maximum  dimension of an irreducible component of $V^i(\sF)$. Moreover, we set:
\begin{gather*}
M_i \; = \; \max \big\{h^i(X,\sF \otimes f^*\alpha) \, | \, \alpha \in V^i(\sF) \, \big\}\\
m_i \; = \; \min \big\{h^i(X,\sF \otimes f^*\alpha) \, | \, \alpha \in V^i(\sF) \, \big\}.
\end{gather*}
Finally, we denote by  $T^i_d$ the set of $d$-torsion points of  $V^i(\sF)$, and by $\tau_d^i = \big| T_d^i\big|$  its cardinality.
We  use  the following lemma in order to bound $\tau^i_d$.

\begin{lem}\label{dtor}
Let $V$ be a complex torus  and  $S= p_0 +B$ be a translate of a subtorus $B\subset V$, and  let $d\geq 1$ be an integer.
If the set of $d$-torsion points of $S$ is not empty, then it consists of exactly $d^{2\dim B}$ elements.
\end{lem}

\begin{proof}
Denote by   $\nu_d(x)=dx$   the multiplication map on $B$.
We notice that if $y= p_0 +  x \in S = p_0 + B$ is a $d$-torsion point, then $dx=-dp_0$. Hence  $x$ is an element of the fiber $\nu_{d}^{-1} ( -d p_0 )$ which consists of exactly $d^{2\dim B}$ elements.
Conversely, if $x \in \nu_{d}^{-1} ( -d p_0 )$, then $y=p_0+x$ belongs to $S$ and it is trivially  a $d$-torsion point.
\end{proof}


\begin{defn}
The locus $V^i(\sF)$ is said \emph{linear} (\emph{resp.} \emph{torsion linear}) if it consists of  a finite union of  translates (\emph{resp.} torsion translates) of  subtori of $\widehat A$.
\end{defn}

\begin{prop}\label{prop1}
If $V^i(\sF)$ is linear, then for all $d\geq 1$ the following inequalities hold:
\begin{itemize}
\item[(i)] $  \tau_d^i  \; \leq \; r_i \, d^{ 2 v_i  },$
\item[(ii)]  $  \sum_{\alpha \in T_d^i} h^i( X, \sF \otimes f^*\alpha) \; \leq \; M_i \, r_i \, d^{2 v_i  }.$
\end{itemize}
 \end{prop}

\begin{proof}
The  proposition follows by Lemma \ref{dtor} and the following inequalities
\begin{equation}\label{mMformula}
m_i \, \tau_d^i \; \leq  \sum_{\alpha \in T_d^i} h^i(X, \sF \otimes f^*\alpha)  \; \leq \; M_i \,\tau_d^i.
\end{equation}
 \end{proof}

 \begin{prop}\label{proplb}
 If $V^i(\sF)$ is torsion linear and $v_i>0$, then 
 $\tau_d^i \geq d^{2 v_i}$ and 
$ \sum_{\alpha \in T_d^i} h^i( X, \sF \otimes f^*\alpha)  \geq m_i d_i^{2v_i}$ 
  for infinitely many $d\geq 1$.
  \end{prop}
\begin{proof}
If $S$ is a component of dimension $v_i$, then it contains $d$-torsion points for infinitely many $d\geq 1$. The result follows by Lemma \ref{dtor} and \eqref{mMformula}.
\end{proof}
There exist upper bounds on the cardinalities  $\tau_d^i$ even if $V^i(\sF)$ is not linear.
 \begin{prop}\label{remray}
There are  positive constants   $a_1,a_2$ such that for all $d\geq 1$ we have:
 \begin{itemize}
\item[(i)] $ \tau_d^i \; \leq \; a_1 \, d^{2 v_i },$
\item[(ii)] $\sum_{\alpha \in T_d^i} h^i(X, \sF \otimes f^*\alpha) \; \leq \; a_2 \, d^{2 v_i }.$
\end{itemize}
\end{prop}

\begin{proof}
We employ the following theorem of Raynaud \cite[p. 327]{Ra83}. Let $Y$ be a closed integral subscheme of a complex abelian variety $V$, and let $T\subset V$ be the set of torsion points. If $T\cap Y$ is dense in $Y$ with respect to the Zariski topology, then $Y$ is a translate of an abelian subvariety by a point of finite order.

Take now the Zariski closure of all the torsion points in $V^i(\sF)$. This is a finite union of irreducible closed subvarieties where 
in each component the torsion points are dense. Hence, by Raynaud's Theorem, each component is a  translate  of an abelian subvariety  of dimension at most $v_i$ by a torsion point.
%
%
%
 \end{proof}
  
 The following theorem is the main result of this section. The equation \eqref{limitcalc} is a generalization of \cite[Theorem 4.1]{Z} in which  the author studies the particular  case of the structure sheaf of a smooth projective variety with respect to the Albanese map.

 \begin{thm}\label{cor1}
If $V^i(\sF)$ is a proper subset of $\widehat A$, 
 then we have 
 \begin{equation}\label{limitcalc}
\lim_{d\to \infty}  \frac{h^i(X_d, \sF_d)}{\deg \varphi_d } \; = \;  0.
\end{equation}
Moreover, if $\sF$ satisfies weak $GV$ with index $p$,  then we have
   $$ \lim_{d\to \infty}  \frac{h^p(X_d, \sF_d)}{\deg \varphi_d}   \; = \;      (-1)^p \chi(\sF).$$ 
 \end{thm}

\begin{proof}
 Denote by $S_d$ the set of all   $d$-torsion points of $\widehat A$ so that 
$$\mu_{d*}\sO_A \; \simeq \; \bigoplus_{\alpha \in S_d} \alpha$$
(\emph{cf}. \cite[Proof of Theorem 4.1]{Z}).
As both $\mu_d$ and $\varphi_d$ are \'etale morphisms,  there are isomorphisms of complexes $\mathbf{R}\mu_{d*}\sO_A \simeq \mu_{d*} \sO_A$ and $\mathbf{R} \varphi_{d*}\sO_{X_d} \simeq \varphi_{d*} \sO_{X_d}$. Hence, by performing the base change of \cite[Lemma 1.3]{BO} along $f\colon X \to A$, we obtain a further decomposition:
\begin{equation*}
\varphi_{d*}  \sO_{X_d}  \; \simeq \; \bigoplus_{\alpha \in S_d} f^*\alpha.
\end{equation*}
Finally, by the projection formula of \cite[Ex. 8.3]{Ha}, we obtain  the following isomorphisms 
$$\varphi_{d*}\sF_d \, \simeq \, \sF \otimes \varphi_{d*} \sO_{X_d} \, \simeq  \, \bigoplus_{\alpha \in S_d}  \big(\sF \otimes  f^*\alpha \big),$$ so that   
\begin{equation}\label{projform} 
 h^i(X_d,\sF_d) \; = \;  
 h^i(X, \varphi_{d*}\sF_d)  \; = \; 
 \sum_{\alpha \in S_d} h^i(X,\sF\otimes f^*\alpha). 
 \end{equation}
 Hence, if $V^i(\sF)=\emptyset$,  then all summands in the right hand side of \eqref{projform} are equal to  zero.
On the other hand, if $V^i(\sF) \neq \emptyset$, then Proposition \ref{remray} yields
 \begin{equation*}
 \sum_{\alpha \in S_d} h^i(X,\sF\otimes f^*\alpha) \;   \leq \; a_i\,  \, d^{2 v_i}
\end{equation*}
for some positive constants $a_i$ which are independent of $d$.
This proves the first claim as  $\deg \varphi_d  = \deg \mu_d = d^{2g}$ and $v_i<g$. 

In order to prove the second claim, we recall  that the Euler characteristic $\chi(\sF)$  is multiplicative under \'etale covers \cite[Proposition 1.1.28]{Laz1}, \emph{i.e.},
$\chi(\sF_d) \, = \, (\deg \varphi_d ) \, \chi(\sF)$. Therefore an application of \eqref{limitcalc} gives
$$\chi(\sF)  \; = \;  \lim_{d\to \infty} \frac{\chi(\sF_d)}{ \deg \varphi_d  } \; = \; \lim_{d\to \infty} (-1)^p  \frac{  h^{p}(X_d, \sF_d)}{ \deg \varphi_d }.$$
 \end{proof}

\begin{rem}\label{ftit}
By Corollary \ref{lemchi}, the   Euler characteristic of a weak $GV$-sheaf with index $p$  satisfies $\chi(\sF) = (-1)^p h^p(X,\sF\otimes f^*\alpha)$, for some line bundle $\alpha$ generic in $\widehat A$. Therefore, if $h^p(X,\sF)$ assumes the least (or generic) value  in  the  set $\{ h^p(X,\sF\otimes f^*\alpha) \, | \, \alpha \in \widehat A \}$, then the computation of $\chi(\sF)$ simplifies to  $$\chi(\sF) \; = \; (-1)^p \, h^p(X,\sF).$$ This is the case if the sheaf $\sF$ satisfies the \emph{Index Theorem with index} $p$ (or $I.T.$ for short), namely that  $V^i(\sF) = \emptyset$ for all $i\neq p$. In fact, by the invariance of the Euler characteristic, it follows  that 
$h^p(X,\sF \otimes f^*\alpha)$ is independent on $\alpha$ and $V^p(\sF) = \widehat A$. 
\end{rem}

\begin{ex}
By Mumford's Index Theorem \cite[\S 16]{Mu08}, any non-degenerate line bundle $L$ on $A$ satisfies the Index Theorem with index $p$, for some $p\in [0, g = \dim A]$ (see Remark \ref{ftit}; moreover note that $p=0$ if and only if $L$ is ample). Therefore, by taking $f={\rm id_A}$, we have
$$\lim_{d\to \infty} \frac{h^p(A,L_d)}{ \deg \mu_d } \; = \; (-1)^p\chi(L) \; = \; (-1)^p\frac{\big( L^g \big) }{g!}.$$
There are examples of  higher rank vector bundles that satisfy the $I.T.$ condition as well, for instance, the class of  non-degenerate simple semi-homogeneous vector bundles on an abelian variety  (\emph{cf}. \cite[Proposition 2.1]{Gr14}).

 
\end{ex}

\section{Limits of Normalized Hodge and Betti Numbers}\label{secbetti}

We denote by  $$h^{p,q}(X) \; = \; \dim_{\C} H^q(X,\Omega_X^p)$$ the Hodge numbers of a smooth projective variety $X$, and  by 
\begin{equation}\label{bettideco}
b_k(X) \; = \; \sum_{p+q=k} h^{p,q}(X)
\end{equation}
  its Betti numbers.

\begin{prop}\label{corbetti}
Let $X$ be a smooth projective variety of dimension $n$ that satisfies the weak generic Nakano vanishing theorem. Then we have
	\[ 
	\lim_{d\to \infty} \frac{h^{p,q}(X_d)}{\deg \varphi_d } \; = \; 
	\begin{cases} 
	(-1)^q \chi(\Omega_X^p)  &\mbox{if }\quad    p+q =  n  \\
	0 &  \mbox{if }\quad   p+q \neq n
	\end{cases}
	\]
	and 
	
	\[ \lim_{d\to \infty} \frac{b_k(X_d)}{\deg \varphi_d } \; = \; 
	\begin{cases} 
	(-1)^n \chi_{\rm top}(X)  &\mbox{if }\quad    k =  n  \\
	0 & \mbox{if }\quad   k  \neq n.
	\end{cases}
	\]
\end{prop}

\begin{proof}
Since $h^{p,q}(X_d) = h^q(X_d, \Omega_{X_d}^p) = h^q(X_d, \varphi_d^* \Omega_X^p)$,  the first statement is an application of Theorem \ref{cor1}.
For $k\neq n$, the second statement follows by the first statement and the equations \eqref{bettideco}. 
For $k=n$, we further observe that
$$\sum_{p=0}^n	 (-1)^{n-p} \chi(\Omega_X^p ) \; =\;  (-1)^n \chi_{\rm top}(X).$$
	 
	%
\end{proof}


\begin{rem}
	Proposition \ref{corbetti} may fail if  the Albanese map is only generically finite onto its image, but not semismall (\emph{cf}. \cite[Remark on p. 6]{JZ}).
	A counterexample is provided by the construction in \cite[\S 3]{gl:gv1} (or \cite[Example 9.1]{Sch}) which we here briefly recall. Let $A$ be an abelian variety of dimension four and let $a_X \colon X \to A$ be the blowup of $A$ along a smooth curve $C \subset A$ of genus $g(C)\geq 2$ with 
	exceptional divisor $E$. Hence $a_X$ is the Albanese map of $X$ and $\delta(a_X)=1$ (see \cite[Example 12.3]{PS13}). By means of the exact sequence $0\to a_X^* \Omega_A \to \Omega_X \to \Omega_{E/C} \to 0$ and the Leray spectral sequence, we deduce that
$$ V^i(\Omega_X) \; = \;  \{\sO_X\}, \quad V^2(\Omega_X) \;  = \; \widehat A, \quad V^3(\Omega_X) \subseteq \{ \sO_X\}, \quad i=0,1,4. $$	
		Hence $\Omega_X$ satisfies weak $GV$ with index $2$, and,  by Theorem \ref{cor1}, we find
	$$\lim_{d\to \infty} \frac{h^{1,2}(X_d)}{ \deg \varphi_d } \; = \; \chi(\Omega_X) \; = \; \chi(\Omega_{E/C}) \; = \;   g(C) -1 \; \neq \; 0.$$
	%
	Moreover, as the loci $V^3(\sO_X) \simeq V^0(\Omega_X^3) \simeq V^1(\omega_X)$  are of codimension at least one (see \cite[Theorem 1]{gl:gv1}), by Theorem \ref{cor1} we have that also the following limit  
	$$\lim_{d\to \infty} \frac{ b_3(X_d) }{ \deg \varphi_d } \; = \; \lim_{d\to \infty} \frac{\sum_{p+q=3} h^{p,q}(X_d)}{\deg \varphi_d}  \; = \; 2 \, \lim_{d\to \infty}
	\frac{h^{1,2}(X_d)}{ \deg \varphi_d } \; = \; 2g(C)-2	$$ is non-zero.
\end{rem}

Now we prove Theorem \ref{main1} of the Introduction. The theorem is a special case of the following more general result, where all the values of 
the defect of semismallness of the Albanese map $\delta(a_X)$ are taken in consideration (see \eqref{defect}). Theorem \ref{main1} is the case $\delta(a_X)=0$. 
First of all, we note  that  the limits \eqref{limhodge} are peculiar to the case $\delta(a_X)=0$, 
and they have   been essentially proved in Proposition \ref{corbetti}.

\begin{thm}\label{thmdefect}
	Let $X$ be a smooth projective variety of complex dimension $n$,  and let $\varphi_d \colon X_d \to X$ be the \'etale covers defined in 	\eqref{diagrintro}. If  the   defect of semismallness of the Albanese map satisfies $\delta(a_X) \leq N$,  then
	for any pair of integers $(p,q) \in  [0,n]^2$ there exists a positive constant $B(p,q)$ such that 
	\begin{equation}\label{degreeh2}
	\frac{h^{p,q}(X_d)}{\deg \varphi_d } \; \leq \; B(p,q) \, d^{ \, - 2 (  \abs{ n-p-q } - N )  } \quad \mbox{for all} \quad d\geq 1.
	\end{equation}
Conversely, if $N\geq 0$ is an integer and  $X$ is a smooth projective variety of dimension $n$ that satisfies both  $\dim \Alb(X) > n$ 
and the bounds in \eqref{degreeh2} for all pairs of indexes $(p,q)\in [0,n]^2$, then the defect of semismallness satisfies $\delta(a_X) \leq N$.
	\end{thm}

\begin{proof}
Let $S_d$ denote the set of $d$-torsion points on $\Alb(X)$. As in \eqref{projform}, we have for all $p$ and $q$ the following equalities
$$h^{p,q}(X_d) \; = \; \sum_{\alpha \in S_d} h^q(X,\Omega_X^p\otimes \alpha).$$ 
By Proposition \ref{prop1}, there exist  positive constants $B = B(p,q)$ such that 
$$\frac{h^{p,q}(X_d)}{ \deg \varphi_d } \; \leq \; B \, d^{2 (\dim V^q(\Omega_X^p) - g )}$$
where $g = \dim \Alb(X)$. 
Moreover,  by Theorem \ref{psthm}, we have $\dim V^q(\Omega_X^p) \leq g - \abs{p+q-n} + \delta(a_X)$ and 
$$\frac{h^{p,q}(X_d)}{ \deg \varphi_d } \; \leq \; B \, d^{-2( \abs{p+q-n} - \delta(a_X) ) }$$ for all $d\geq 1$. This shows one implication.

Assume now that $\dim \Alb(X) > n$ and that the bounds \eqref{degreeh2} hold.
Moreover, assume by contradiction that $\delta(a_X) \geq N +1$. By Theorem \ref{psthm}, there exists a pair
$(p_0,q_0) \in [0,n]^2$ such that $\codim V^{q_0} ( \Omega_X^{p_0} ) =  \abs{ n - p_0 - q_0 } - \delta(a_X)$. 
Then $\dim V^{q_0}(\Omega_X^{p_0}) = \dim \Alb(X) - \abs{n-p_0-q_0} +\delta(a_X) >0$, and by
Proposition \ref{proplb} we have 
$$ \frac{h^{p_0,q_0}(X_d)}{\deg \varphi_d} \; \geq \; A \, d^{-2 ( \abs{n-p_0-q_0} - \delta(a_X) ) } \quad \mbox{ for infinitely many  } \quad d\geq 1$$
for some positive  constant $A$ independent of $d$
(note that the loci $V^q(\Omega_X^p)$ are torsion linear by \cite[Corollary 19.2]{Sch}).
For $d\gg 0$, this contradicts the bounds  \eqref{degreeh2} when $(p,q)=(p_0,q_0)$.
\end{proof}

\begin{proof}[Proof of Corollary \ref{maincor1}.]
The first statement of the corollary is an application of Theorem \ref{main1} and \eqref{bettideco}. On the other hand, the second point is again Proposition \ref{corbetti}.
\end{proof}

\section{Limits of Normalized Plurigenera}\label{secplurigenera}

In this subsection, we apply Theorem \ref{cor1} to the pluricanonical bundles $\omega_X^{\otimes m}$ ($m\geq 1$) of a smooth projective variety $X$. 
We set $p_g(X) = P_1(X) = h^0(X,\omega_X)$ for the geometric genus of $X$, and 
$$P_m(X) \; = \;  h^0(X,\omega^{\otimes m}_X), \quad m\geq 2,$$ for the plurigenera.

In the following proposition we fix a morphism $f\colon X \to A$ to an abelian variety.
\begin{prop}\label{corpm} Let $X_d$ be the fiber product between $f\colon X \to A$ and $\mu_d$ as  in the commutative diagram \eqref{fundamental}.
Then for any integer $m\geq 1$ we have 
	\begin{equation}\label{limpluri} 
	\lim_{d\to \infty} \frac{P_m(X_d)}{\deg \varphi_d} \; = \;   \chi(f_*\omega_X^{\otimes m}).
	\end{equation}
	Moreover, if $f\colon X \to A$ is generically finite onto its image, then 
	$$\lim_{d\to \infty} \frac{ p_g(X_d) }{\deg \varphi_d} \; = \;   \chi(\omega_X).$$
	 
\end{prop}

\begin{proof}
	By \cite[Theorem 1.10]{popa+schnell:direct} the sheaves $f_*\omega_X^{\otimes m}$ satisfy $GV$ for all $m\geq 1$.
	Hence, by Theorem \ref{cor1}, 	we have 
	$$\lim_{d\to \infty} \frac{ h^0(A, \mu_d^* f_* \omega^{\otimes m}_X ) }{\deg \varphi_d}  \; = \; \chi(f_*\omega_X^{\otimes m}).$$
	We observe that by base change,  together with the fact that $h^0(A,f_{d*} \sG) = h^0(X_d,\sG)$ for any coherent sheaf $\sG$ on $X_d$, we have the equalities
	$$h^0(X_d, \omega_{X_d}^{\otimes m}) \; = \; h^0(X_d, \varphi_d^* \omega_X^{\otimes m}) \; = \; h^0(A, f_{d*} \varphi_d^* \omega_X^{\otimes m})  \; = 
	\; h^0(A, \mu_d^* f_*\omega_X^{\otimes m}).$$
	
	The second  claimed limit follows by the  Grauert--Riemenschneider  Vanishing  \cite[Theorem 4.3.9]{Laz1}, which yields $\chi(f_*\omega_X)=\chi(\omega_X)$ if $f$ is generically finite onto its image.  
\end{proof} 

For $m\geq 2$ there are two cases where one can improve the results of Proposition \ref{corpm}.
The first is the case of smooth projective varieties of general type. Indeed, Koll\'ar in  \cite[Proposition 9.4]{KP} shows the multiplicativity of the higher plurigenera under any \'etale map, so that $\frac{ P_m(X_d) }{ \deg \varphi_d }$ are constants and trivially
$$ \lim_{d\to \infty} \frac{ P_m(X_d) }{ \deg \varphi_d } \; = \; P_m(X) \quad \mbox{for all}\quad m\geq 2.$$
The second is the case of    the Albanese map $f = a_X \colon X \to \Alb(X)$.
With a slight abuse of notation, we denote by $I\colon X \to Z$ a non-singular representative of the Iitaka fibration of 
$X$. Moreover, we  set 
$$q(I) \;=\; q(X)-q(Z) \; = \;  h^0(X,\Omega_X) - h^0(Z,\Omega_Z)$$ for the difference of the irregularities.

\begin{prop}\label{plurinu}
	Let $X_d$ be the fiber product between $a_X$ and $\mu_d$ as in \eqref{diagrintro}, and fix an integer $m\geq 2$. 	Then  there exists a positive  constant $M$ such that 
	$$\frac{P_m(X_d)}{\deg \varphi_d}   \; \leq \; M \, d^{-2\,q(I)}\quad \mbox{ for all } \quad d\geq 1.$$
	Moreover we have
	\[ \lim_{d\to \infty} \frac{P_m(X_d)}{\deg \varphi_d}   \; = \; 
	\begin{cases} 
	P_m(X)  & \mbox{if }\quad   q(I)=0 \\
	0 &  \mbox{if } \quad  q(I)>0.
	\end{cases}
	\]
\end{prop}

\begin{proof}
By \cite[Theorem 11.2 (b)]{HPS}, for  each $m\geq 2$ there exist line bundles  $\alpha_1, \ldots , \alpha_t \in \Pic^0(X)$ of finite order such that
	$$V^0(\omega_X^{\otimes m} ) \; = \; \bigcup_{j=1}^t \big(\alpha_j + \Pic^0(Z) \big).$$   
	Hence we have $\dim V^0(\omega_X^{\otimes m}) = q(Z)$.
Moreover, as $P_m(X_d) = \sum_{\alpha \in S_d} h^0(X, \omega_X^{\otimes m} \otimes \alpha)$ (where $S_d$ is the set of $d$-torsion points on $\Alb(X)$),   
by Proposition \ref{prop1} there exists a positive constant $M>0$ such that   the first claim holds.
This also shows that the limit $\lim_{d\to \infty} \frac{P_m(X_d)}{\deg \varphi_d}$ vanishes  for $q(I)>0$.
In order to complete the proof, thanks to  Proposition \ref{corpm}, 
we only need to calculate  the Euler characteristic $\chi(a_{X*}\omega_X^{\otimes m})$.  As $a_{X*}\omega_X^{\otimes m}$ satisfies $GV$, by Lemma \ref{lemchi} we find that 
\begin{equation}\label{eqchi}	
	\chi(a_{X*}\omega_X^{\otimes m}) = h^0 (\Alb(X), a_{X*}\omega_X^{\otimes m} \otimes \alpha) = h^0(X, \omega_X^{\otimes m} \otimes \alpha)
\end{equation}	
	 for  a generic element $\alpha \in \Pic^0(\Alb(X)) \simeq \Pic^0(X)$. However, 
	 if $q(I)=0$, then  by  \cite[Theorem 11.2 (a)]{HPS} we have $V^0(\omega_X^{\otimes m}) = \Pic^0(X)$ and moreover the quantities $h^0(X,\omega_X^{\otimes m} \otimes \alpha)$ are independent of $\alpha$.
	 
\end{proof}

\begin{rem}
Smooth projective varieties of general type fall within the class $q(I)=0$.	Instances of  varieties with $q(I)=0$, but which are   not of general type,  are provided by non-isotrivial elliptic surfaces fibered over smooth projective curves $\Sigma_g$ of genus $g\geq 2$. Indeed, given an elliptic surface $p \colon X \to \Sigma_g$, one can show that the corresponding morphism $P\colon \Alb(X)\to \Alb(\Sigma_g)$ is an isomorphism if and only if the elliptic fibration is \emph{not} isotrivial. For more details, we refer to \cite[Chapter IX]{Bea}. Higher-dimensional examples may be constructed in the same fashion.
\end{rem}

In analogy to Koll\'ar's result \cite[Proposition 9.4]{KP},  the previous proposition suggests that the higher plurigenera ought to be multiplicative under \'etale morphisms  also in the more general case $q(I)=0$. We confirm this expectation for the \'etale covers induced via base change by the isogenies of $\Alb(X)$. In \cite[Corollary 12.2]{LPS}, the reader may notice a further property that shows how varieties with $q(I) = 0$ behave like varieties of general type. Indeed, the sheaves  $a_{X*}\omega_X^{\otimes m}$ satisfy $I.T.$ with index $0$  for all $m\geq 2$ as soon as  $q(I)=0$ (\emph{cf}. Remark \ref{ftit}).

\begin{thm}\label{multiplication}
Let $X$ be a smooth projective variety with $q(I)=0$, and let $Y$ be the fiber product between $a_X$ and an isogeny $\mu \colon B \to \Alb(X)$, as in the following cartesian diagram:
	\begin{equation*}
	\centerline{ \xymatrix@=32pt{
				Y \ar[r]^{\widetilde a \,\,\,} \ar[d]^{\varphi} & B \ar[d]^{\mu}  \\
				X \ar[r]^{a_X\,\,\,\,\,}  & \Alb(X).}}
	\end{equation*}
Then for all $d\geq 1$ and $m\geq 2$  we have  $P_m(Y) = (\deg \varphi) P_m(X)$. 
\end{thm}

\begin{proof}
	The proof follows the general strategy of \cite[Theorem 11.2]{HPS} and \cite[Theorem 11.2.23]{Laz2}. This goes as follows.
	As in the proof of Proposition \ref{plurinu}, the Iitaka fibration $I$ induces a surjective morphism $a_I \colon \Alb(X) \to \Alb(Z)$  with connected fibers such that the diagram
	\begin{equation*}\label{fundamentalItake}
	\centerline{ \xymatrix@=32pt{
			X \ar[r]^{a_X} \ar[d]^{I} & \Alb(X) \ar[d]^{a_I}  \\
			Z  \ar[r]^{a_Z}  & \Alb(Z) \\}} 
	\end{equation*}
	commutes (\emph{cf}. \cite[Proposition 2.1]{HP}). 
Therefore $a_I$ is an isomorphism as $q(I)=0$.  Fix now an integer $m\geq 2$ 
	and let $\sI \big(\norm*{\omega_X^{\otimes (m-1)}} \big)$ be the asymptotic multiplier ideal sheaf as defined in \cite[Definition 11.1.2]{Laz2}.
	Moreover set $g=a_I\circ a_X$ and define the sheaf
	$$\sH \; = \; g_*  \Big( \omega_X^{\otimes m}   \otimes  \sI \big (\norm*{\omega_X^{\otimes (m-1)}}  \big) \Big).$$
	Since $g$ factors though $I$, we have a linear equivalence relation $tK_X\sim g^*H+E$ where $H$ is an ample divisor on $\Alb(Z)$, $E$ is an effective divisor, and $t\gg 0$ is a sufficiently large integer.
This implies, as proved in the course of the proof of \cite[Theorem 11.2]{HPS},  that 
$$H^0(\Alb(Z),\sH) \; = \;  H^0(X,\omega_X^{\otimes m}) \quad \mbox{ and } \quad H^i(\Alb(Z), \sH) = 0  \mbox{ for all }  i>0.$$  Thus we have  $P_m(X) = \chi(\sH)$.  

Now, consider the Stein factorization $s \colon Y \to S$ of the composition $I\circ\varphi \colon Y \to Z$. 
As the general fiber of $s$ has Kodaira dimension equal to zero, 
by \cite[Remark 2.1.35]{Laz1} $s$ factors through a non-singular representative of the  
 Iitaka fibration of $Y$, which, with a slight abuse of notation, we denote it by $I_Y \colon Y \to W$. 
Define $\widetilde g=a_I \circ \mu \circ \widetilde a $.  
	As $\widetilde g$ factors though $I_Y$, we can write
\begin{equation}\label{eq:iitaka}	
	 \widetilde t K_Y \sim \widetilde g^*   \widetilde H + \widetilde E
	 \end{equation}
	for some ample line bundle $\widetilde H$  on $\Alb(Z)$, effective divisor $\widetilde E$ on $Y$,  
	and  large integer $\widetilde t\gg 0$. 
%
By defining the sheaves
	$$\widetilde \sG \, = \, \widetilde g_*\Big( \omega_{Y}^{\otimes m} \otimes \sI \big(\norm*{\omega_{Y}^{\otimes (m-1)} } \big) \Big), \, \; \widetilde \sH \, = \, \widetilde a_{*}\Big( \omega_{Y}^{\otimes m} \otimes \sI \big( \norm*{\omega_{Y}^{\otimes (m-1)} } \big) \Big),$$
the relation \eqref{eq:iitaka}  ensures that 
\begin{equation}\label{eq:vanishingk}
	H^0 \big(\Alb(Z), \widetilde \sG \big) = H^0 \big(Y,\omega_Y^{\otimes m} \big) \quad \mbox{ and }\quad H^i \big(\Alb(Z), \widetilde \sG \big) \, = \, 0
	 \; \mbox{ for } \; i>0, 
	 \end{equation}
again as shown in the argument of  the proof of  
 \cite[Theorem 11.2]{HPS}.	\footnote{Even if not explicitly stated, 
 the proof of   \cite[Theorem 11.2]{HPS} actually proves the isomorphism and vanishings in 
 \eqref{eq:vanishingk} for any morphism from $Y$ to an abelian variety that factors 
 through the Iitaka fibration of $Y$.}
	 We conclude  that  $P_m(Y)= \chi(\widetilde \sG) = \chi(\widetilde \sH)$ as in addition there are isomorphisms $H^i(B, \widetilde \sH) \simeq  H^i(\Alb(Z), \widetilde \sG)$ for all $i\geq 0$ (recall that $\mu$ is \'etale). 
Finally, we note that $\chi(\widetilde  \sH) = \chi \big(\mu^* \big(a_I^{-1} \big)_*  \sH \big)$. Indeed, by base change and \cite[Theorem 11.2.16]{Laz2},  we obtain the following isomorphisms
$$\widetilde \sH  \; \simeq \; \widetilde a_* \varphi^* \Big(\omega_X^{\otimes m} \otimes \sI \big( \norm*{ \omega_X^{\otimes (m-1)} } \big) \Big)\;  \simeq \; \mu^* \big(a_{I}^{-1}\big)_*\sH. $$
To conclude, we note that  
\[
\chi \Big(\mu^* \big(a_I^{-1}\big)_* \sH \Big) \;  = \;  (\deg \mu) \,  \chi \Big(\big(a_I^{-1}\big)_* \sH \Big) \; = \; (\deg \mu) \, \chi(\sH)
\]
as $a_I\colon\Alb(X) \to \Alb(Z)$ is an isomorphism.
\end{proof}
 
\begin{rem}[Higher direct images and  multiplier ideal sheaves]
One can apply Theorem \ref{cor1} to other classes of sheaves that satisfy the generic vanishing condition of Definition \ref{defgv}. In this direction, the paper \cite{pareschi+popa:gv} contains  several examples of $GV$-sheaves. As an example, by keeping the notation of  \eqref{fundamental}, Theorem \ref{cor1} and \cite[Theorem 5.8]{pareschi+popa:gv} give 
$$\lim_{d\to \infty} \frac{ h^0(A , R^if_{d*}\omega_{X_d} ) }{\deg \mu_d} \; = \; \chi(R^i f_* \omega_X)\quad \mbox{ for any } \quad i\geq 0.$$
Moreover,  Theorem \ref{cor1} in combination with  \cite[Corollary 5.2]{pareschi+popa:gv} give the following statement. Suppose that the Albanese map $a_X \colon X \to \Alb(X)$ is generically finite onto its image, and let $L$ be a line bundle with non-negative Kodaira dimension. With notation as in  \eqref{diagrintro}, we have  
$$\lim_{d\to \infty} \frac{ h^0 \big(X_d, \omega_{X_d}  \otimes L_d \otimes \sI \big( \norm*{L_d} \big )\big) }{ \deg \varphi_d } \; = \; \chi\big(\omega_X \otimes L \otimes \sI \big( \norm*{L} \big)\big).$$

\end{rem}

\section{Applications to $L^2$-Cohomology}\label{secl2}

In order to define $L^2$-Betti numbers we follow the reference \cite{Luck02}.   
Let $G$ be a discrete group, and let $M$ be a co-compact free proper $G$-manifold without boundary endowed with a $G$-invariant Riemannian metric. 
Define the space of smooth $L^2$-integrable harmonic  $k$-forms 
\[
\mathcal{H}^k_{(2)} (M) \; = \; \big\{ \, \omega\in\Omega^k (M) \;  | \;  \Delta_d \omega = 0, \; \int_M \omega \wedge *\omega  <\infty \big\}
\]
where $*$ is the Hodge star operator and $\Delta_d = d d^* + d^* d$ is the Hodge-Laplacian operator.
By \cite[Section 1.3.2]{Luck02}, the spaces $\mathcal{H}^k_{(2)}(M)$ are finitely generated Hilbert modules over the von Neumann algebra $\mathcal{N}(G)$ of  $G$.  We define the $L^2$\emph{-Betti numbers}  $b^{(2)}_k \big(M; \mathcal{N}(G) \big)$ of $(M,G)$ as  the von Neumann dimension of the $\mathcal{N}(G)$-modules   $\mathcal{H}^k_{(2)}(M)$:
\[
b^{(2)}_k \big(M;\mathcal{N}(G) \big) \; \stackrel{{\rm def}}{=} \; \dim_{\mathcal{N}(G)}\mathcal{H}^k_{(2)}( M ).
\] 
The $L^2$-Betti numbers assume values in the extended interval $[0,\infty]$ of the real numbers, and $b^{(2)}_k \big(M, \mathcal{N}(G) \big) \in [0,\infty)$ if the action of $G$ is co-compact. 

Finally, in order to define the $L^2$\emph{-Hodge numbers} $h^{(2)}_{p,q} \big(M;\mathcal{N}(G) \big)$ of $(M,G)$,  we define
\[
\mathcal{H}^{p,q}_{(2)} (M) \; =  \; \big\{ \, \omega\in\Omega^{p,q} (M) \;  | \;  \Delta_{\overline \partial}  \omega = 0, \; \int_M \omega \wedge *\omega  < \infty \big\},
\]
where $\Delta_{\overline \partial} = \overline \partial  \,  \overline \partial^* + \overline \partial^*  \,  \overline \partial$ is the $\overline \partial$-Laplacian, and set
\[
h^{(2)}_{p,q} \big(M;\mathcal{N}(G) \big) \; \stackrel{{\rm def}}{=} \; \dim_{\mathcal{N}(G)}\mathcal{H}^{p,q}_{(2)}(M).
\] 
By \cite[Chapter 11]{Luck02}, there is  a $L^2$-Hodge decomposition which gives
\begin{equation}\label{l2hodge}
b^{(2)}_k \big(M;\mathcal{N}(G) \big) \; = \; \sum_{p+q=k} h^{(2)}_{p,q}\big(M;\mathcal{N}(G) \big).
\end{equation}

\subsection{(Non-)Vanishing of $L^2$-Betti numbers}

Let $X$ be a smooth projective variety of dimension $n$, and let $a_X \colon X \to \Alb(X)$ be the Albanese map. Moreover set $g = \dim \Alb(X)$. The \emph{universal Albanese cover} $\overline \pi \colon \overline X \to X$ is defined as the  pullback of $a_X$ via the topological universal cover $\C^g \to \Alb(X)$ (\emph{cf}. \cite[Section 3.2]{DCL19}). 
We set 
$$\Gamma = \pi_1(X), \quad \overline{\Gamma} = \pi_1( \overline{X} ), \quad  G =\Gamma / \overline{ \Gamma} \quad \mbox{ and }\quad A=\Alb(X).$$ 

\begin{thm}\label{corl2}
 If $X$ satisfies the weak generic Nakano vanishing theorem, then the $L^2$-Betti numbers of $\overline X$ are 
\[  b_k^{(2)} \big(\overline{X}; \mathcal{N}(G)  \big) \; = \; 
\begin{cases} 
      (-1)^n \chi_{\rm top}(X)  &\mbox{if }\quad    k =  n  \\
           0 & \mbox{if }\quad   k  \neq n.
   \end{cases}
\]
In particular, we have $\mathcal{H}_{(2)}^{p,q}(\overline{X}) = 0$ if $p+q\neq n$.
\end{thm}

\begin{proof}

Consider the following cartesian diagram induced inductively by the multiplication maps $\mu_d$ via the base change:
$$
 \centerline{ \xymatrix@=32pt{
\cdots \ar[r] & Y_{d+1} \ar[r] \ar[d] & Y_d \ar[d] \ar[r] & Y_{d-1} \ar[r] \ar[d]  & \cdots \ar[r] &  Y_2 \ar[d] \ar[r] &  X  \ar[d]^{a_X}  \\
 \cdots  \ar[r] & A  \ar[r]^{\mu_{d+1}}  & A  \ar[r]^{\mu_d} & A \ar[r]^{\mu_{d-1}}  & \cdots \ar[r] & A \ar[r]^{\mu_2} &  A.\\}} 
 \noindent
$$
By using  the notation of the commutative diagram \eqref{fundamental}, we immediately realize that $Y_d\simeq X_{d!}$. We set $\Gamma_d = \pi_1(Y_d)$. 
By \cite[Section 3.1]{DCL19}, together with the proof of \cite[Lemma 3.2]{DCL19}, the homomorphism $a_{X\#} \colon \pi_1(X) \to \pi_1(A)$ is surjective
and   the varieties $Y_d$ satisfy
$${\rm ker} \big(a_{X\#} \colon \pi_1(X) \to \pi_1(A) \big) \; = \; \bigcap_{d=1}^{\infty} \Gamma_d$$
(\emph{i.e.}, in the terminology of  \cite[Lemma 3.2]{DCL19}, the isogenies $r_i$ can be chosen as  the multiplication maps $\mu_d$). 
Moreover, the universal Albanese  cover $\overline \pi \colon \overline X \to X$ is identified to  the regular cover associated to the normal separable subgroup ${\rm ker}(a_{X\#})$. 
Therefore $\overline \Gamma = {\rm ker}(a_{X\#})$ and there are isomorphisms
$$Y_d \; \simeq \; \bigslant{\overline X}{G_d} \quad \mbox{ where } \quad  G_d \; \stackrel{{\rm def}}{=} \; \bigslant{\Gamma_d}{\overline \Gamma}.$$
%
As the sequence $\{ G_d  \}_{d\geq 1}$ is an inverse  system of normal subgroups such that  $\cap_{d\geq 1} G_d = \{ 1\}$, L\"uck's Approximation Theorem \cite[Theorem 13.3]{Luck02} and \cite[Example 1.32]{Luck02} yield
$$b_k^{(2)} \big(\overline{X}; \mathcal{N}(G) \big)   \; = \;
\lim_{d\to \infty} b_k^{(2)}\big( Y_d; \mathcal{N}(G / G_d) \big) \; = \; \lim_{d\to \infty} \frac{b_k(Y_d )}{ [\Gamma : \Gamma_d ]} 
 \; = \; \lim_{d\to \infty} \frac{b_k(X_{d!})}{\deg \mu_{d!}}.$$  At this point,  the first statement of the theorem is an application of Proposition \ref{corbetti}.  On the other hand, the second follows by the fact that the von Neumann dimension of a Hilbert module is zero if and only if the module itself is trivial  (see \cite[Theorem 1.12 (1)]{Luck02}), and   the $L^2$-Hodge decomposition \eqref{l2hodge}.
\end{proof}

The following non-vanishing  result was proved by Gromov in the case of topological universal covers of  K\"ahler hyperbolic manifolds (\emph{cf}. \cite[Section 2]{Gro} and \cite[Theorem 11.35]{Luck02}).

\begin{cor}\label{cornonvan}
Let $X$ be as in Theorem \ref{corl2}. If $\chi(\Omega_X^{p})\neq 0$, then we have 
$$\mathcal{H}_{(2)}^{p,n-p}(\overline{X}) \; \neq \; 0.$$ Moreover, if $\chi(\omega_X) \neq 0$, then there exists a nontrivial holomorphic $L^2$-integrable $n$-form on the universal Albanese cover $\overline X$.
\end{cor}
\begin{proof}
We use the notation of the previous proof.
First of all we observe that by Proposition \ref{corbetti} we have  the inequalities 
$$\limsup_{d\to \infty} \frac{h^{p,n-p} (Y_{d}) }{ \deg \mu_{d!}}  \; \geq \lim_{d \to \infty} \frac{ h^{p,n-p}(X_{d!})}{\deg \mu_{d!}}  \; = (-1)^{n-p} \chi(\Omega_X^p) \;  >  \; 0$$
Thus the result follows by Kazhdan's inequality \cite[Theorem 2]{K} (\emph{cf}. also \cite[p. 6-7]{JZ}): 
$$ h^{(2)}_{p,n-p} \big(\overline X, \mathcal{N}(G)  \big)\; \geq \; \limsup_{d\to \infty} \frac{h^{p,n-p} (Y_{d}) }{ \deg \mu_{d!}}.$$
The second statement is proved  as in  \cite[Corollary 11.36]{Luck02}. In other words,  if we have a non-zero form $\omega \in \sH^{n,0}_{(2)}(\overline X)$, then $\Delta_{\overline \partial} \omega = 0$ and $\overline \partial \omega = 0$. This means that  $\omega$ is holomorphic. 
\end{proof}

 
 \section{Appendix: Coverings of  Varieties with Unbounded Irregularity}\label{unbounded}
 Let $Y$ be a smooth projective variety satisfying  $q(Y) = h^{1,0}(Y)= \dim \Alb(Y)>0$ and $H^2(Y,\Z)_{ \rm tor}=0$. 
 We provide sufficient and necessary conditions for the irregularities $Q= \{ q(Y_i) \}_{i=1}^{\infty}$ of a series of  coverings  $\pi_i \colon Y_i \to Y$ induced 
 by the multiplication maps on $\Alb(Y)$ to diverge as $\deg \pi_i \to \infty$. 
This problem has been already  addressed  in the literature.
For instance, by the recent work of Vidussi \cite[Lemma 1.3]{Vid17} and Stover \cite[Theorem 3]{Sto17}, the irregularity of any unramified abelian cover
of the  Cartwright--Steger surface\footnote{The Cartwright--Steger surface $S$  is a complex hyperbolic surface with \emph{minimal} Euler characteristic $\chi(S)=3$, and non-trivial first Betti number. It was computationally discovered in \cite{CS} during the classification of fake projective planes.  We refer to \cite{Yeung2} for an in depth study of its geometry.} is equal to one. 
On the other hand, it is very easy to construct towers of coverings with unbounded irregularities. 
 
Turning to details, let $X$ be a smooth projective variety of dimension $n$ and $a_X \colon X \to \Alb(X)$ be the Albanese map.
The multiplication maps $\mu_d \colon \Alb(X) \to \Alb(X),$ $\mu_d(x)=dx$ induce via base-change unramified covers $a_d \colon X_d \to X$. We use the term \emph{fibration} to mean a surjective morphism of varieties with connected fibers. The following result builds upon \cite[Corollaire 2.3]{B}.

 \begin{thm}\label{qfibr}
Suppose that  $\limsup_{d\to \infty} q(X_d)=\infty$. 
 	Then  $X$ admits a fibration $p\colon X \to C$ onto a smooth curve of genus $g$
 	such that either $g\geq 2$, or $g=1$ and the fibration admits  two multiple fibers whose multiplicities are not coprime. If in addition $H^2(X,\Z)_{ \rm tor}=0$, then  the converse holds.
 \end{thm}

 \begin{proof}
 Let $S_d$ be the set of $d$-torsion points of $\Pic^0(X)$.
 	The irregularity $q(X_d)=h^1 (X_d, \sO_{X_d})$ can be computed with the techniques  of Theorem  \ref{cor1}:
 	\begin{equation}\label{eqq}
 	q(X_d) \; = \; q(X) \; + \; \sum_{\alpha \in S_d, \,\alpha \neq  \sO_X} h^1(X, \alpha).
 	\end{equation}
 	First of all we prove that $\limsup_{d\to \infty} q(X_d)=\infty$ if and only if  there exists a positive-dimensional component of the Green--Lazarsfeld locus
 	$$V^{n-1}(\omega_X) \; \simeq \;  V^1(\sO_X) \stackrel{{\rm def}}{=} \{ \, \alpha \in \Pic^0(X) \; | \; h^1(X,\alpha)>0 \, \}. $$
 	In fact, if 
 	$\limsup_{d\to \infty} q(X_d) = \infty$, then by \eqref{eqq}  there are infinitely many distinct elements of $V^1(\sO_X)$. As $V^1(\sO_X)$ is an algebraic variety, these elements must form one irreducible component.
 	On the other hand, if
 	$v_1 = \dim V^1(\sO_X)>0$, then, by Proposition \ref{proplb}, $V^1(\sO_X)$\footnote{By \cite[Corollary 19.2]{Sch}, the locus $V^1(\sO_X)$ is a finite union of torsion translates of abelian subvarieties of $\Pic^0(X)$.}  contains at least $d^{2v_1}$ $d$-torsion points for infinitely many $d\geq 1$. 
 	  Hence $q(X_d) \geq q(X) + d^{2v_1} -1$ and the claim follows.
 	
 	Let now  $\Pic^{\tau}(X)$ be  the variety that parameterizes isomorphism classes of holomorphic line bundles on $X$ with torsion first Chern class.
 	By the work of \cite[Theorem 0.1]{gl:gv2} and \cite[Corollaire 2.3]{B}, the irreducible components of $V^{n-1}(\omega_X)$ are related to fibrations over smooth projective curves. More precisely, any positive-dimensional irreducible component $S \subset V^{n-1}(\omega_X)$  is a component of the group $$\Pic^{\tau}(X,p) \; \stackrel{{\rm def}}{=} \; \ker\big( i^*  \colon  \Pic^{\tau}(X) \to \Pic^{\tau}(F) \big)$$ 
 	for some fibration $p  \colon X \to C$ over a smooth projective curve of genus $g\geq 1$ with general fiber $i \colon F \hookrightarrow X$. It follows that $\dim S = g$. Moreover, if $g=1$, then by	\cite[Corollaire 2.3]{B} we have $S \neq p^*\Pic^0(C)$. Therefore, by \cite[Remarque 2.4]{B}, $p$  must posses at least two multiple fibers whose multiplicities are not coprime (\emph{cf}. also   \cite[Exercise 10.3]{Sch}). This proves one of the implications.
 	
 Let now $p\colon X \to C$ be a fibration onto a smooth projective curve of genus $g\geq 1$. 
  	For the other direction, we note that if $g\geq 2$, then, by pulling-back line bundles from  $\Pic^0(C) = V^0(\omega_C)$,  the fibration $p$ gives rise to an irreducible  component of $V^{n-1}(\omega_X)$ (\emph{cf}. \cite[Lemma 6.3]{lombardi:invariants}). 
We now prove that we reach the same conclusion even if $g=1$, $H^2(X,\Z)_{\rm tor}=0$, and the condition on the multiple fibers is verified. 	In fact, the condition on the  multiple fibers  implies that the group $\Gamma^{\tau}(p) \simeq \Pic^{\tau}(X,p) / p^*\Pic^0(C)$ of the  connected components of $\Pic^{\tau}(X,p) $  is non-trivial (\emph{cf}. \cite[Proposition 1.5, Remarque 2.4]{B} and \cite[Exercise 10.3]{Sch}). By \cite[\S 1.6]{B}, as $H^2(X,\Z)_{ \rm tor}=0$, the group  $\Gamma^{\tau}(p)$ is identified with  the group $\Gamma^0(p)$ of the connected components of the following group  $$\Pic^0(X,p) \;  \stackrel{{\rm def}}{ = } \; \Pic^{\tau}(X,p) \cap \Pic^0(X).$$ Therefore  $\Pic^0(X,p)$ contains a connected component different from the neutral component $ p^*\Pic^0(C)$, which, again by \cite[Corollaire 2.3]{B}, it is contained in $V^{n-1}(\omega_X)$.
 \end{proof}

 \begin{rem}\label{remvid}
 Let $S$ be the Cartwright--Steger surface. It follows from \cite{CS} that $H_1(S,\Z)$ is torsion free. By the universal coefficient theorem, we know that $H^2(S,\Z)_{\rm tor}=0$. Moreover, the Albanese map has no multiple fibers (\emph{cf}. \cite[Main Theorem]{Yeung2}). Then Theorem \ref{qfibr} implies that the unramified abelian covers of $S$ have  bounded irregularities (however much more is true for the surface $S$, \emph{cf}. again \cite{Vid17} and \cite{Sto17} for optimal statements).
 \end{rem}
 
 The argument of Proposition \ref{qfibr} extends, in a weaker form, to all Hodge numbers of type $h^{n,i}(X)$ with $i>0$.
 \begin{prop}
 	If $\limsup_{d\to \infty}  h^{n,i}(X_d)  \; = \; \infty$, 
 	then there exists a fibration of $X$ onto a normal projective variety $Y$ of dimension $0<\dim Y \leq n-i$ such that any smooth model of $Y$ is of  maximal Albanese dimension.
 \end{prop}
 \begin{proof}
 	Thanks to  a calculation similar  to \eqref{eqq},  we can construct an irreducible component  $S \subset V^i(\omega_X)$ of positive dimension. 
By \cite[Theorem 0.1]{gl:gv2}, this component induces a fibration of $X$ onto a variety with the desired  properties.
 \end{proof}
 
%
%
%
%
%
%
%

\end{document}